\documentclass[10pt,reqno]{amsart}

\usepackage{pifont}
\usepackage{mathrsfs}
\usepackage{amscd,color}
\usepackage{amsmath}
\usepackage{latexsym}
\usepackage{amsfonts}
\usepackage{amssymb}
\usepackage{amsthm}
\usepackage{graphicx}
\usepackage{amsmath,amsfonts}
\usepackage{CJK}
\usepackage[linkcolor=blue,citecolor=blue]{hyperref}

\makeatletter
\newcommand{\Extend}[5]{\ext@arrow0099{\arrowfill@#1#2#3}{#4}{#5}}
\makeatother
\let\pa=\partial

\def\eps\epsilon

\def\cA{{\cal A}}

\def\C{\mathop{\bf C\kern 0pt}\nolimits}
\def\DD{\mathop{\bf D\kern 0pt}\nolimits}
\def\K{\mathop{\bf K\kern 0pt}\nolimits}
\def\N{\mathop{\bf N\kern 0pt}\nolimits}
\def\Q{\mathop{\bf Q\kern 0pt}\nolimits}

\newcommand{\beq}{\begin{equation}}
\newcommand{\eeq}{\end{equation}}
\newcommand{\ben}{\begin{eqnarray}}
\newcommand{\een}{\end{eqnarray}}
\newcommand{\beno}{\begin{eqnarray*}}
\newcommand{\eeno}{\end{eqnarray*}}

\def\R{\mathop{\mathbb R\kern 0pt}\nolimits}

\newtheorem{theorem}{Theorem}[section]
\newtheorem{proposition}[theorem]{Proposition}
\newtheorem{lemma}[theorem]{Lemma}

\theoremstyle{remark}

\begin{document}

 \title[Schr\"odinger-Choquard equation]{ \bf Scattering for the non-radial  focusing inhomogeneous
nonlinear Schr\"odinger-Choquard equation}

\author{Chengbin Xu}%
\address{The Graduate School of China Academy of Engineering Physics,
 P. O. Box 2101,\ Beijing,\ China,\ 100088;
}%
\email{xcbsph@163.com}%

\begin{abstract}
In this paper, we study the long-time behavior of global solutions to the Schr\"odinger-Choquard equation
$$i\pa_tu+\Delta u=-(I_\alpha\ast|\cdot|^b|u|^{p})|\cdot|^b|u|^{p-2}u.$$
   Inspired by Murphy, who gave a simple proof of scattering for the non-radial inhomogeneous NLS, we prove scattering theory below the ground state for the intercritical case in energy space without radial assumption.
\end{abstract}
\maketitle

\begin{center}
 \begin{minipage}{100mm}
 \end{minipage}
 \end{center}

\setcounter{section}{0}\setcounter{equation}{0}
\section{Introduction}

\noindent

 Considering the initial value problem (IVP), also called the Cauchy problem for the inhomogenous nonlinear Schr\"odinger-Choquard equation
 \begin{align}\label{INLC}
 \begin{cases}
   &i\pa_tu+\Delta u=-(I_\alpha\ast|\cdot|^b|u|^{p})|\cdot|^b|u|^{p-2}u,\ \ \ \  t\in\R,\ x\in\R^N\\
   &u(0,x)=u_0(x) \in H^1(\R^N)
 \end{cases}
 \end{align}
 where $u:\R\times\R^N\rightarrow\mathbb{C}$ and $N\geq3$. The inhomogeneous term is $|\cdot|^b$ for some $b<0$. The Riesz-potential is defined on $\R^N$ by
 $$I_\alpha:=\frac{\Gamma(\frac{N-\alpha}{2})}{\Gamma(\frac{\alpha}2)\pi^{\frac{N}2}2^{\alpha}|\cdot|^{N-\alpha}}:
 =\frac{\mathcal{K}}{|\cdot|^{N-\alpha}},\ \ 0<\alpha<N.$$
 And
 $$1+\frac{2+\alpha+2b}{N}<p<1+\frac{2+\alpha+2b}{N-2}.$$
We also assume that
\begin{align}\label{condition}
 \min\{2+\alpha+2b,N+b,N+4b+2\alpha,4+\alpha+2b-N\}>0.
\end{align}

The class of solutions to \eqref{INLC} is left invariant by the scaling
\begin{equation}\label{equ:scalam}
  u_{\lambda}(t,x)=\lambda^{\frac{2+2b+\alpha}{2(p-1)}}u(\lambda^2 t,\lambda x),
\end{equation}
which is also invariant in
 $\dot H^{s_c}(\R^N)$ norm with $s_c=\frac{N}2-\frac{2+2b+\alpha}{2(p-1)},$. Thus, we call that the equation \eqref{INLC} is $\dot H^{s_c}(\R^N)$ critical.
Moreover,  equation \eqref{INLC} conserves the mass, defined by
$$M(u):=\int_{\R^N}|u|^2dx,$$
and the energy, defined as the sum of the kinetic and potential energies:
$$E(u):=\int_{\R^N}\Big[\frac12|\nabla u|^2-\frac1{p}(I_\alpha\ast|\cdot|^b|u|^b)|x|^{b}|u(x)|^{p}\Big]dx.$$

 The particular case $b=0$, Eq.\eqref{INLC} becomes the nonlinear generalized Hartree equation. In \cite{FY}, Feng-Yuan have studied the Cauchy problem for the generalized  Hartree equation in the energy subcritical. Miao-Xu-Zhao \cite{MXZ3,MXZ4} have studied the well-posedness issues of the Hartree equation which corresponded to the particular case $p=2$. Saanouni \cite{Sa} proved the scattering theory for the radial case via concentration-compactness roadmap.

For the Eq.\eqref{INLC}, Alharbi-Saanouni \cite{AS} studied the local theory and finite time blow-up. In \cite{Sa2}, Saanouni proved the scattering theory with the radial setting. In this work, we extend Saanouni's work to non-radial case via a new approach established by Murphy \cite{Mu}.

As we known, the equation \eqref{INLC} admits a global but non-scattering solution
$$u(t,x)=e^{it}Q(x),$$
where $Q$ is the ground state, i.e., the solution to elliptic equation
$$-\Delta Q+Q-(I_\alpha\ast|\cdot|^b|Q|^p)|x|^b|Q|^{p-2}Q=0.$$
We refer to Alharbi-Saanouni \cite{AS} about the existence of the ground state for the above elliptic equation. The uniqueness is still open.

Now, we state our main result.
\begin{theorem}\label{theorem}
  Let $p\geq2$ and $b, \alpha$ satisfy \eqref{condition}.  Suppose that $u_0\in H^1(\R^N)$ satisfies
  $$M(u_0)^{1-s_c}E(u_0)^{s_c}<M(Q)^{1-s_c}E(Q)^{s_c}\quad \text{and}\quad \|u_0\|_{L_x^2}^{1-s_c}\|\nabla u_0\|_{L_x^2}^{s_c}<\|Q\|_{L_x^2}^{1-s_c}\|\nabla Q\|_{L_x^2}^{s_c}.$$
  Then, there exists a unique global solution $u$ to \eqref{INLC}. Moreover, the global solution $u$ scatters in the sense that there exists $u_\pm\in H^1(\R^N)$ such that
  \begin{equation}\label{equ:gsusc}
    \lim_{t\to\pm\infty}\|u(t,x)-e^{it\Delta}u_\pm\|_{H^1_x(\R^N)}=0.
  \end{equation}
\end{theorem}

Theorem \ref{theorem} was established first in the radial case by Saanouni \cite{Sa2}. For the non-radial case, the non-radial approach by Dodson-Murphy \cite{DM} fails since the lack of Galilean invariant. However, we have a key observation that the nonlinearity has a decay factor $|x|^{b}$ at infinity. Hence, this allow us treat the non-radial setting as the radial case (for more details, see \cite{Mu}).

This paper is organized as follows: In Section $2$, we recall some basic estimates to prove some main tools and establish a new scattering criterion. In Section $3$, we prove a Morawetz estimate, which in turn implies `energy evacuation' as $t\to\infty$. Thus, these two ingredients quickly complete Theorem \ref{theorem}. Finally, a Morawetz identity is proved in Appendix.


We conclude the introduction by giving some notations which
will be used throughout this paper. We always use $X\lesssim Y$ to denote $X\leq CY$ for some constant $C>0$.
Similarly, $X\lesssim_{u} Y$ indicates there exists a constant $C(u)$ depending on $u$ such that $X\leq C(u)Y$.
We also use the notation $\mathcal{O}$: e.g. $A=\mathcal{O}(B)$ indicates $A\leq CB$ for constant $C>0$.
The derivative operator $\nabla$ refers to the spatial  variable only.
We use $L^r(\mathbb{R}^N)$ to denote the Banach space of functions $f:\mathbb{R}^N\rightarrow\mathbb{C}$ whose norm
$$\|f\|_r:=\|f\|_{L^r}=\Big(\int_{\mathbb{R}^N}|f(x)|^r dx\Big)^{\frac1r}$$
is finite, with the usual modifications when $r=\infty$. For any non-negative integer $k$,
we denote by $H^{k,r}(\mathbb{R}^N)$ the Sobolev space defined as the closure of smooth compactly supported functions in the norm $\|f\|_{H^{k,r}}=\sum_{|\alpha|\leq k}\|\frac{\partial^{\alpha}f}{\partial x^{\alpha}}\|_r$, and we denote it by $H^k$ when $r=2$.
For a time slab $I$, we use $L_t^q(I;L_x^r(\mathbb{R}^N))$ to denote the space-time norm
\begin{align*}
  \|f\|_{L_{t}^qL^r_x(I\times \R^N)}=\bigg(\int_{I}\|f(t,x)\|_{L^r_x}^q dt\bigg)^\frac{1}{q}
\end{align*}
with the usual modifications when $q$ or $r$ is infinite, sometimes we use $\|f\|_{L^q(I;L^r)}$ or $\|f\|_{L^qL^r(I\times\mathbb{R}^N)}$ for short.

\section{Preliminaries}

\noindent

Let us start this section by introducing the notation used throughout the paper. We recall some Strichartz estimates associated to the linear Schr\"odinger propagator.

We say the pair $(q,r)$ is $\dot{H}^s$-admissible, if it  satisfies the condition
$$\frac{2}{q}=\frac{N}{2}-\frac{N}{r}-s,\ \ 2\leq q,r\leq\infty, \text{and}\ \ N\geq3.$$

For $s\in(0,1)$, We define the sets $\Lambda_s$:
$$\Lambda_s:=\left\{(q,r)\ \text{is}\ \dot{H}^s\text{-admissible};\Big(\frac{2N}{N-2s}\Big)^+\leq r\leq \Big(\frac{2N}{N-2}\Big)^-\right\},$$
and
$$\Lambda_{-s}:=\left\{(q,r)\ \text{is}\ \dot{H}^{-s}\text{-admissible};\Big(\frac{2N}{N-2s}\Big)^+\leq r\leq \Big(\frac{2N}{N-2}\Big)^-\right\},$$
and let $\Lambda_0$ denote the $L^2$-admissible. Here, $a^-$ is a fixed number slightly smaller than $a$ ($a^-=a-\epsilon$, where $\epsilon$ is small enough). $a^+$ can be defined by the same way.

Next, we define the following Strichartz norm
$$\|u\|_{S(\dot{H}^s,I)}=\sup_{(q,r)\in \Lambda_s}\|u\|_{L_t^qL_x^r(I)}$$
and dual Strichartz norm
$$\|u\|_{S^{'}(\dot{H}^{-s},I)}=\inf_{(q,r)\in \Lambda_{-s}}\|u\|_{L_t^{q^{'}}L_x^{r^{'}}(I)},$$
where $q'$ denotes the dual exponent to $q$, i.e. the solution to $\tfrac1q+\frac1{q'}=1.$
 If $I=\R$, $I$ is omitted usually.

 Now, let us recall some results about Strichartz estimates \cite{Ca,Fo,KT} and Hardy-Littlewood-Sobolev's inequality \cite{Li}.
\begin{lemma}
  Let $0\in I\subset\R$, the following statement hold
  \begin{enumerate}
    \item[$(i)$]$($linear estimate$)$
    $$\|e^{it\Delta}f\|_{S(\dot{H}^s)}\leq C\|f\|_{\dot{H}^s};$$
    \item[$(ii)$]$($nonlinear estimate I$)$
  $$\left\|\int_{0}^{t}e^{i(t-s)\Delta}g(\cdot,s)ds\right\|_{S(L^2,I)}\leq C\|g\|_{S^{'}(L^2,I)};$$
    \item[$(iii)$]$($nonlinear estimate II$)$
  $$\left\|\int_{0}^{t}e^{i(t-s)\Delta}g(\cdot,s)ds\right\|_{S(\dot{H}^s,I)}\leq C\|g\|_{S^{'}(\dot{H}^{-s},I)}.$$

  \end{enumerate}

\end{lemma}

\begin{lemma}
  Let $N\geq3$, $0<\lambda<N$ and $1<r,s<\infty<\infty$ and $f\in L^r, g\in L^s$. If $\frac1r+\frac1s+\frac{\lambda}N=2$, then
  $$\iint_{\R^N\times\R^N}\frac{f(x)g(y)}{|x-y|^\lambda}dxdy\leq C(N,s,\lambda)\|f\|_{L^r_x}\|g\|_{L^s_x}.$$
\end{lemma}
Next, we prove some interpolation estimates for nonlinearities, which plays an important role in proving scattering theory.
\begin{lemma}[Nonlinear estimate]\label{Non-e1}
Let $N\geq 3$ , $b, \alpha$ satisfy \eqref{condition} and $p\geq2$. Then there exists $\theta\in (0,2(p-1))$ sufficiently small such that
\begin{enumerate}
  \item[$(i)$] $\|(I_{\alpha}\ast|\cdot|^{b}|u|^p)|\cdot|^b|u|^{p-2}u\|_{S'(\dot{H}^{s_c})}\lesssim \|u\|_{L_t^\infty H_x^1}^\theta\|u\|_{S(\dot{H}^{s_c})}^{2p-1-\theta}$;

  \item[$(ii)$]$\|(I_{\alpha}\ast|\cdot|^{b}|u|^p)|\cdot|^b|u|^{p-2}u\|_{S'(L^2)}\lesssim \|u\|_{L_t^\infty L^2}^\theta\|u\|_{S(\dot{H}^{s_c})}^{2(p-1)-\theta}\|u\|_{S(L^2)}$;
   \item[$(iii)$]$\|\nabla(I_{\alpha}\ast|\cdot|^{b}|u|^p)|\cdot|^b|u|^{p-2}u\|_{S'(L^2)}\lesssim \|u\|_{S(\dot{H}^{s_c})}^{2(p-1)-\theta}(\|u\|_{L^\infty H^1_x}^{\theta}\|\nabla u\|_{S(L^2)}+\|u\|_{L^\infty H^1_x}^{1+\theta})$.

\end{enumerate}

\end{lemma}
\begin{proof}
  In view of the singular factor $|x|^b$ in the nonlinearity, we frequently divide our analysis in two region. Let $B_1(0)$ denote the unit ball of radius $1$ and center in origin, and $B_1^c(0)$ be $\R^N\setminus B_1(0)$.

  We introduce the parameters
  \begin{align}\label{exponent-1}
    \hat{r}=\frac{2(p-1)N(2p-\theta)}{2(p-1)(N+2b+\alpha)-\theta(\alpha+2b+2)}.
  \end{align}
  Choosing $\hat{q},\hat{a},\tilde{a}$ and $\theta$ such that $(\hat{q},\hat{r})\in\Lambda_0$, $(\hat{a},\hat{r})\in\Lambda_{s_c}$ and $(\tilde{a},\hat{r})\in\Lambda_{-s_c}$. These exponents obey the relations
  \begin{align}\label{scale-1}
    \frac1{\tilde{a}'}=\frac{2p-1-\theta}{\hat{a}}\ \ and\ \ \frac1{\hat{q}'}=\frac{2(p-1)-\theta}{\hat{a}}+\frac1{\hat{q}}.
  \end{align}

  We first prove $(i)$, using the pair $L_t^{\tilde{a}'}L_x^{\hat{r}'}$. Let $r_1$ be chosen later and define $\mu$ so that
  \begin{align}
    1+\frac{\alpha}{N}=\frac{2}{\mu}+\frac{2p-\theta}{\hat{r}}+\frac1{r_1}.
  \end{align}
  Then, $A\in\{B_1(0),B_1^c(0)\}$, using the Hardy-Littlewood-Paley inequality to estimate
  \begin{align}
    \|(I_{\alpha}\ast|\cdot|^{b}|u|^p)|\cdot|^b|u|^{p-2}u\|_{L^{\hat{r}'}}\lesssim\||x|^b\|_{L^{\mu}(A)}^2\|u\|_{L^{\theta r_1}}^{\theta}\|u\|_{L^{\hat{r}}}^{2p-1-\theta}.
  \end{align}
  Using the scaling relation above, we derive
  $$\frac{2N}{\mu}+2b=\frac{\theta(2+\alpha+2b)}{2(p-1)}-\frac{N}{r_1}.$$
  Thus, if $A=B_1(0)$ we choose $r_1$ so that $\theta r_1=\frac{2N}{N-2}$. On the other hand, if $A=B_1^c(0)$, we choose $\theta r_1=2$. In both case, we have Sobolev embedding that $H^1\subset L^{\theta r_1}$. Thus, we take the $L_t^{\tilde{a}'}$-norm, apply H\"older's inequality, and use the above scaling relation \eqref{scale-1} to obtain
  $$\|(I_{\alpha}\ast|\cdot|^{b}|u|^p)|\cdot|^b|u|^{p-2}u\|_{L_t^{\tilde{a}'}L_x^{\hat{r}'}}\lesssim \|u\|_{L_t^\infty H_x^1}^\theta\|u\|_{L_t^{\hat{a}}L_x^{\hat{r}}}^{2p-1-\theta}.$$
  The estimates of $(ii)$ is similar. In this case, we use the second relation of \eqref{scale-1} to get
  $$\|(I_{\alpha}\ast|\cdot|^{b}|u|^p)|\cdot|^b|u|^{p-2}u\|_{L_t^{\hat{q}'}L_x^{\hat{r}'}}\lesssim \|u\|_{L_t^\infty H^1_x}^\theta\|u\|_{L_t^{\hat{a}}L_x^{\hat{r}}}^{2 (p-1)-\theta}\|u\|_{L_t^{\hat{q}}L_x^{\hat{r}}}.$$

  Consider now $(iii)$. We choose the exponents $\bar{q},\bar{r},\bar{a}$
  $$\bar{r}=\frac{4(p-1)N(2p-1-\theta)}{2(p-1)(N+2+4b+2\alpha)-\theta(2+\alpha+2b)},$$
  such that $(\bar{q},\bar{r})\in\Lambda_{0},(\bar{a},\bar{r})_{\Lambda_{s_c}}$. Since
  \begin{align*}
    |\nabla(I_{\alpha}\ast|\cdot|^b|u|^p)|x|^b|u|^{p-2}u|\lesssim& (I_{\alpha}\ast|\cdot|^{b-1}|u|^p)|x|^b|u|^{p-1}\\
    &+(I_{\alpha}\ast|\cdot|^b|u|^{p-1}|\nabla u|)|x|^b|u|^{p-1}\\
    &+(I_{\alpha}\ast|\cdot|^b|u|^p)|x|^{b-1}|u|^{p-1}\\
    &+(I_{\alpha}\ast|\cdot|^b|u|^p)|x|^b|u|^{p-2}|\nabla u|,
  \end{align*}
 by using Hardy-Littlewood-Sobolev's inequality, we get
  \begin{align*}
  \|\nabla(I_{\alpha}\ast|\cdot|^b|u|^p)|x|^b|u|^{p-2}u\|_{L_x^{\frac{2N}{N+2}}}\lesssim& \||x|^b\|_{L_x^{r_1}(A)}^2(\||x|^{-1}|u|^p\|_{L_x^{r_2}}+\||u|^{p-1}\nabla u\|_{L_x^{r_2}})\||u|^{p-1}\|_{L_x^{r_3}}\\
  &+\||x|^b\|_{L_x^{r_1}(A)}^2\||u|^p\|_{L_x^{r_2}}(\||x|^{-1}|u|^{p-1}\|_{L_x^{r_3}}+\||u|^{p-2}\nabla u\|_{L_x^{r_3}}),
  \end{align*}
  with the following scaling relation:
  $$1+\frac{\alpha}{N}=\frac{N-2}{2N}+\frac{2}{r_1}+\frac{1}{r_2}+\frac{1}{r_3}.$$
  And, we choose $l$ such that
  \begin{align*}
    \frac2{r_1}+\frac{2b}{N}=l\ \ \text{and}\ \ l:=
    \begin{cases}
   &\frac{\theta(1-s_c)}{N},\ \ \text{if}\  A=B,\\
   &-\frac{\theta s_c}{N},\ \ \ \ \text{if}\  A=B^c.
 \end{cases}
  \end{align*}
 Since $1<\frac{2N}{N+2+4b+2\alpha}<N$, if we choose $\theta$ small enough, we conclude that $\||x|^{b}\|_{L_x^{L^{r_1}}(A)}<\infty$ and that $1<r_2,r_3<N$. Hence, by Hardy's inequality, we have
  $$\||x|^{-1}f\|_{L_x^{r_2}}\leq \|\nabla f\|_{L_x^{r_2}}.$$
  Now, by splitting
  $$\frac1{r_2}+\frac1{r_3}=\left(\theta\left(\frac12-\frac{s_c}{N}\right)-l\right)\left(\frac1{r_4}\right)+\frac{2(p-1)-\theta}{\bar{r}}+\frac1{\bar{r}},$$
  $$\frac1{r_3}=\frac{p-1}{\bar{r}},$$
we can get $2\leq \theta r_4\leq 2N/(N-2)$. Thus, using H\"older's inequality and Sobolev inequality
\begin{align}
  \|\nabla(I_{\alpha}\ast|\cdot|^b|u|^p)|x|^b|u|^{p-2}u\|_{L_t^2L_x^{\frac{2N}{N+2}}}\lesssim \|u\|_{L_t^\infty H_x^1}^\theta\|u\|_{L_t^{\bar{a}}L_x^{\bar{r}}}^{2(p-1)-\theta}\|\nabla u\|_{L_t^{\bar{a}}L_x^{\bar{r}}},
\end{align}
which completes lemma.
\end{proof}
\subsection{Scattering criterion}To show Theorem \ref{theorem}, we first establish a scattering criterion by following that in \cite{Mu}.
\begin{lemma}[Scattering Criterion]
Let $b,\alpha, p$ satisfy the condition of Theorem \ref{theorem}. Suppose $u$ is the global  solution to \eqref{INLC} satisfying
$$\|u\|_{L_t^\infty H_x^1}\leq E.$$
There exists $\epsilon=\epsilon(E)>0$ and $R=R(E)>0$ such that if
\begin{align}\label{scattering-condition}
\liminf_{t\to\infty}\int_{|x|\leq R}|u(t,x)|^2dx\leq \epsilon^2,
\end{align}
then $u$ scatters forward in time.
\end{lemma}
\begin{proof}
By Duhamel's formula, we have
  $$u=e^{i(t-T_0)\Delta}u(T_0)+i\int_{T_0}^{t}e^{i(t-s)\Delta}F(u(s))ds$$
  where $F(u)=(I_{\alpha}\ast|\cdot|^b|u|^p)|x|^b|u|^{p-2}u.$

  By Lemma \ref{Non-e1} and continuity argument, we need to show
  $$\|e^{i(t-T_0)\Delta}u(T_0)\|_{L_t^{\hat{a}}L_x^{\hat{r}}([T_0,\infty))}\ll 1.$$
  where the exponent $(\hat{a},\hat{r})$ is as in Lemma \ref{Non-e1}.

  Next, we use the Duhamel formula to write
  \begin{align*}
  e^{i(t-T_{0})\Delta}u(T_{0})=e^{it\Delta}u_0-iG_{1}(t)-iG_{2}(t),
\end{align*}
where
\begin{align*}
  G_{j}(t):=\int_{I_{j}}e^{i(t-s)\Delta}F(u(s))ds, ~j=1,2,
\end{align*}
here $I_{1}=[0,T_{0}-\epsilon ^{-\eta}], I_{2}=[T_{0}-\epsilon ^{-\eta},T_{0}].$

Choosing $T_0$ large enough, we have
$$
  \|e^{it\Delta}u_0\|_{L_t^{\hat{a}}L_x^{\hat{r}}([T_0,\infty))} \ll1.
$$
 It remains to show
\begin{align*}
  \|G_{j}(t)\|_{L_t^{\hat{a}}L_x^{\hat{r}}([T_0,\infty))} \ll 1,\quad \text{~ for ~} j=1,2.
\end{align*}

\textbf{Estimation  of $G_{1}(t)$}: We may use the dispersive estimate, Hardy-Littlewood-Sobolev's inequality and H\"older's inequality, we have
\begin{align*}
  \left\|G_1(t)\right\|_{L_t^{q_0}L_x^{\hat{r}}}&\lesssim \left(\int_{T_0}^{\infty}\left(\int_{I_1}|t-s|^{-\frac{2}{q_0}-1}\|u\|_{H_x^1}^{2p-1}ds\right)^{q_0}dt\right)^{\frac1{q_0}}\\
  &\lesssim_E \left(\int_{T_0}^{\infty}|t-T_0+\epsilon^{-\theta}|^{-2}dsdt\right)^{\frac1{q_0}}\\
  &\lesssim_E \epsilon^{\frac{\eta}{q_0}},
\end{align*}
where let $q_0$ satisfy $(q_0,\hat{r})\in \Lambda_1$.

On the other hand, we may rewrite $ G_1$ as
\begin{equation}\label{F1}
G_1(t)=e^{i(t-T_0+\epsilon^{-\eta})\Delta}u(T_0-\epsilon^{-\eta})-e^{it\Delta}u_0.
\end{equation}
By $(\hat{q},\hat{r})\in \Lambda_{0}$, then
$$\frac1{\hat{a}}=\frac{1-s_c}{\hat{q}}+\frac{s_c}{q_0}.$$
Using Strichartz estimates and \eqref{F1},
 we have
$$\|G_1(t)\|_{L_t^{q_0}L_x^{\hat{r}}}\lesssim 1.$$
Thus, by interpolation, we get
$$\|G_1(t)\|_{L_t^{\hat{a}}L_x^{\hat{r}}([T_0,\infty))}\lesssim \epsilon^{\frac{\eta s_c}{q_0}}.$$
\textbf{Estimation  of $G_{2}(t)$}: By Strichartz's estimates, we have
\begin{align*}
 \|G_2(t)\|_{L_t^{\hat{q}}L_x^{\hat{r}}([T_0,\infty))}\lesssim & \|F(u)\|_{L_t^{\tilde{a}'}L_x^{\hat{r}'}(I_2)}\\
 \lesssim&\|\chi_RF(u)\|_{L_t^{\tilde{a}'}L_x^{\hat{r}'}(I_2)}+\|(1-\chi_R)F(u)\|_{L_t^{\tilde{a}'}L_x^{\hat{r}'}(I_2)},
\end{align*}
where $\chi_R$ is a smooth cutoff to $\{x:\;|x|\leq R\}.$
Using  Hardy-Littlewood-Sobolev's inequality and H\"older's inequality, we get
$$\|\chi_RF(u)\|_{L^{\hat{r}'}}\lesssim \|u\|_{H^1}^{2(p-1)}\|\chi_Ru\|_{L^{\hat{r}}}\lesssim_E\|\chi_Ru\|_{L^2}^{\beta},$$
where $\beta$ satisfies $1/\hat{r}=\beta/2+(1-\beta)/2^*$ with $2^\ast=\tfrac{2N}{N-2}.$ On the other hand, we have
\begin{align*}
 \|(1-\chi_R)F(u)\|_{L^{\hat{r}'}}\lesssim&\||x|^b\|_{L^{\mu_1}(A)}\||x|^b\|_{L^{\mu_2}(|x|\geq R)}\|u\|_{L^{\theta r_1}}^{\theta}\|u\|_{L^{\hat{r}}}^{2p-1-\theta},
\end{align*}
 where the exponents have the scaling relation
  $$\frac{N}{\mu_1}+\frac{N}{\mu_2}+2b=\frac{\theta(2+\alpha+2b)}{2(p-1)}-\frac{N}{r_1}.$$
Since $\frac{\theta(N-2)}2<\frac{\theta(2+\alpha+2b)}{2(p-1)}<\frac{\theta N}2$, there exists $\mu_2>\frac{N}{-b}$, if $A=B_1(0)$ we can choose $\theta r_1=\frac{2N}{N-2}$, or else $A=B_1^c(0)$ we can choose $\theta r_1=2$. In both case, we have $H^1\subset L^{\theta r_1}$ and $|x|^b\in L^{\mu_1}(A).$
Thus, we have
$$\|(1-\chi_R)F(u)\|_{L^{\hat{r}'}}\lesssim R^{\mu_2b+N}\|u\|_{H^1}^{2p-1}.$$

Let $T_0$ be large enough. By the assumption \eqref{scattering-condition} and identity $\pa_t|u|^2=-2\nabla\cdot Im(\bar{u}\nabla u)$, together with integration by parts and Cauchy-Schwartz, we deduce
$$\left|\pa_t\int_{I_2}\chi_R|u|^2ds\right|\lesssim \frac{1}{R}.$$
Choosing $R\geq \max\{\epsilon^{-(2+\theta)},\epsilon^{\frac1{N+\mu_2b}}\}$, we find
$$\|\chi_Ru\|_{L_t^\infty L_x^2(I_2\times \R^N)}\lesssim \epsilon,$$
and
$\|F(u)\|_{L^{\hat{r}'}}\lesssim_E \epsilon^{\beta}.$

Then, let $\eta=\frac{\hat{a}'\beta}2$, we bound
\begin{align*}
  \|G_2(t)\|_{L_t^{\hat{q}}L_x^{\hat{r}}([T_0,\infty))}\lesssim&
  \|F(u)\|_{L_t^{\tilde{a}'}L_x^{\hat{r}'}(I_2)}\\
  \lesssim& |I_2|^{1/\hat{a}'}\epsilon^{\beta}\\
  \lesssim&\epsilon^{\frac{\beta}2},
\end{align*}
which completes the proof of lemma.
\end{proof}
\subsection{ Local theory and Variational analysis}The local theory for \eqref{INLC} is standard via Strichartz estimates and the fixed point argument. For any $u_0\in H^1(\R^N)$, there exists a unique maximal-lifespan solution $u$. This solution belongs to $C_tH_x^1(I_{\max}\times \R^N)$ and conserves the mass and energy. Because the nonlinear term is $H^1$-subcritical, we have an $H^1$ blow-up criterion. In particular, if $u$ is uniformly bounded in $H^1$, then it is global. For more details, we refer the reader to \cite{AS,Sa}.

\noindent

We briefly review some of the variational analysis related to the ground state $Q$. For more details, see \cite{AS,Sa}.

The ground state $Q$ optimizes the sharp Gagliardo-Nirenberg inequality:
$$\int_{\R^N}(I_\alpha\ast|\cdot|^b|u|^p)|x|^b|u|^pdx\leq C_0\|u\|_{L^2_x}^A\|\nabla u\|_{L_x^2}^B,$$
where $B:=Np-N-\alpha-2b=2(p-1)s_c+2$, $A:=2p-B=2(p-1)s_c$. And $Q$ satisfies $\|\nabla Q\|^{2}=\frac{B}{A}\|Q\|^{2}$,
$$\int_{\R^N}(I_\alpha\ast|\cdot|^b|Q|^p)|x|^b|Q|^pdx=\frac{2p}{B}\|\nabla Q\|^2,$$
thus $$\frac{2p}{B}\leq C_0\|u\|_{L^2_x}^{2(p-1)(1-s_c)}\|\nabla u\|_{L^2_x}^{2(p-1)s_c}.$$

In the spirit of Dodson-Murphy \cite{DM-2017-PAMS, DM}, we can obtain the following coercivity, which can be founded in \cite{Sa2}.
\begin{lemma}[Coercivity I]\label{CoeI}
Suppose that $u_0\in H^1(\R^N)$ satisfies
  $$M(u_0)^{1-s_c}E(u_0)^{s_c}<(1-\delta)M(Q)^{1-s_c}E(Q)^{s_c}\quad \text{and}\quad \|u_0\|_{L_x^2}^{1-s_c}\|\nabla u_0\|_{L_x^2}^{s_c}<\|Q\|_{L_x^2}^{1-s_c}\|\nabla Q\|_{L_x^2}^{s_c}.$$ Then, there exists $\delta'=\delta'(\delta)>0$ so that
$$\|u(t)\|_{L_x^2}^{1-s_c}\|\nabla u(t)\|_{L_x^2}^{1-s_c}<(1-\delta')\|Q\|_{L_x^2}^{1-s_c}\|\nabla Q\|_{L_x^2}^{1-s_c}$$
for all $t\in I$, where $u:I\times\R^N\rightarrow\mathbb{C}$ is the maximal-lifespan solution to \eqref{INLC}. In particular, $I=\R$ and $u$ is uniformly bounded in $H^1$.
\end{lemma}
\begin{lemma}[Coercivity II] \label{Co-II}
Suppose $\|f\|_{L_x^2}^{1-s_c}\|\nabla f\|_{L_x^2}^{s_c}\leq(1-\delta)\|Q\|_{L_x^2}^{1-s_c}\|\nabla Q\|_{L_x^2}^{s_c}$. Then there exists $\delta'>0$ so that
$$\int_{\R^N}|\nabla f|^2dx-\frac{B}{2p}\int_{\R^N}(I_\alpha\ast|\cdot|^b|u|^p)|x|^b|u|^pdx\geq \delta' \int_{\R^N}|\nabla f|^2dx. $$
\end{lemma}
\begin{proof}
  Using the identity
  $$\int_{\R^N}|\nabla f|^2dx-\frac{B}{2p}\int_{\R^N}(I_\alpha\ast|\cdot|^b|u|^p)|x|^b|u|^pdx=\frac{B}{2}E(f)-\frac{B-2}{2}\|\nabla f\|_{L_x^2}^2.$$
  By the sharp Gagliardo-Nirenberg inequality
  \begin{align*}
   E(f)&\geq \|\nabla f\|_{L_x^2}\left[1-\frac{C_0}{p}\|f\|_{L_x^2}^{A}\|\nabla f\|_{L_x^2}^{B-2}\right]\\
    &\geq \|\nabla f\|_{L_x^2}^2\left[1-\frac{(1-\delta)C_0}{p}\|Q\|_{L_x^2}^{A}\|\nabla Q\|_{L_x^2}^{B-2}\right]\\
   &\geq \frac{B-2(1-\delta)^{2(p-1)}}{B}\|\nabla f\|_{L_x^2}^2.
 \end{align*}
  Thus
  $$\int_{\R^N}|\nabla f|^2dx-\frac{B}{2p}\int_{\R^N}(I_\alpha\ast|\cdot|^b|u|^p)|x|^b|u|^pdx\geq \delta'\|\nabla f\|_{L_x^2}^2.$$
\end{proof}

\section{Proof of Theorem 1.1}

\noindent

In this section, we turn to prove Theorem \ref{theorem}. Assume that $u$ is a solution to \eqref{INLC} satisfying the hypothesis of Theorem \ref{theorem}. It follows from Lemma \ref{CoeI} that $u$ is global, and
$$\sup_{t\in\R}\|u(t)\|_{L_x^2}^{1-s_c}\|\nabla u(t)\|_{L_x^2}^{1-s_c}<(1-\delta')\|Q\|_{L_x^2}^{1-s_c}\|\nabla Q\|_{L_x^2}^{1-s_c}.$$

 First, we need a lemma that gives Lemma \ref{Co-II} on large balls, so that we can exhibit the necessary coercivity. Let $\chi(x)$ be radial smooth function such that
 \begin{equation*}
  \chi(x)=\left\{
    \begin{aligned}
    &1,\ \  |x|\leq \frac12,\\
    &0,\ \  |x|>1.
 \end{aligned}\right.
 \end{equation*}
 Set $\chi_R(x):=\chi(\frac xR)$ for $R>0$.
 \begin{lemma}[Coercivity on ball, \cite{Sa2}]
   There exists $R=R(\delta, M(u), Q)>0$ sufficiently large that
   $$\sup_{t\in\R}\|\chi_R u\|_{L_x^2}^{1-s_c}\|\nabla(\chi_Ru)\|_{L_x^2}^{s_c}\leq (1-\delta)\|Q\|_{L_x^2}^{1-s_c}\|\nabla Q\|_{L_x^2}^{s_c}.$$
   In particular, by Lemma \ref{Co-II}, there exits $\delta'>0$ such that
$$\int_{\R^N}|\nabla(\chi_Rf)|^2dx-\frac{B}{2p}\int_{\R^N}(I_\alpha\ast|\cdot|^b|\chi_Ru|^p)|x|^b|\chi_Ru|^pdx\geq \delta' \int_{\R^N}|\nabla(\chi_Rf)|^2dx. $$
 \end{lemma}

Let $R\gg 1$ to be chosen later. We take $a(x)$ to be a radial function satisfying
\begin{eqnarray}\label{weight-function}
a(x)=
\begin{cases}
|x|^2;& |x|\leq R\\
3R|x|;& |x|>2R,
\end{cases}
\end{eqnarray}
and when $R<|x|\leq 2R$, there holds
\begin{align*}
  \partial_{r}a\geq 0,\partial_{rr}a\geq 0\quad and \quad |\partial^{\alpha}a| \lesssim R|x|^{-|\alpha|+1}.
\end{align*}
Here $\partial_{r}$ denotes the radial derivative. Under these conditions, the matrix $(a_{jk})$ is non-negative.
It is easy to verify that
\begin{eqnarray*}
\begin{cases}
a_{jk}=2\delta_{jk},\quad \Delta a=2N,\quad \Delta \Delta a=0,& |x|\leq R,\\
a_{jk}=\frac{3R}{|x|}[\delta_{jk}-\frac{x_{j}x_{k}}{|x|^2}],\quad \Delta a=\frac{3(N-1)R}{|x|},\quad \Delta \Delta a=\frac{-3(N-1)(N-3)R}{|x|^3},& |x|>2R.
\end{cases}
\end{eqnarray*}

\begin{lemma}[Morawetz identity]
Let $a:\R^N\rightarrow\R$ be a smooth weight. Define
$$M(t)=2{\rm Im}\int_{\R^N}\bar{u}\nabla u\cdot\nabla adx.$$
Then
\begin{align*}
 \frac{d}{dt}M_a(t)=&\int_{\R^N}(-\Delta\Delta a)|u|^2dx+4\int_{\R^N}a_{jk}{\rm Re}(\pa_j\bar{u}\pa_ku)dx\\
 &-\bigg(2-\frac{4}{p}\bigg)\int_{\R^N}\Delta a|x|^b|u|^p(I_\alpha\ast|\cdot|^b|u|^p)dx\\
 &+\frac{4b}p\int_{\R^N}\nabla a\cdot\frac{x}{|x|^2}(I_\alpha\ast|\cdot|^b|u|^p)|u|^pdx\\
 &-\frac{2\mathcal{K}(N-\alpha)}{p}\int_{\R^N}\int_{\R^N}(\nabla a(x)-\nabla a(y))\cdot\frac{x-y}{|x-y|^{N-\alpha+2}}|y|^b|u|^p|x|^b|u|^pdydx.
\end{align*}
where subscripts denote partial derivatives and repeated indices are summed.
\end{lemma}
\begin{proposition}[Morawetz estimates]
Let $T>0$ and choosing $R=R(\delta,M(u_0),Q)$ sufficiently large, then
$$\frac1{T}\int_{0}^{T}\int_{|x|<R}|u(t,x)|^{\frac{2N}{N-2}}dxdt\lesssim_{u,\delta}\frac{R}{T}+\frac1{R^{-b}}.$$
In particular, there exists a sequence of times $t_n,R_n\rightarrow\infty$ so that
$$\lim_{n\to\infty}\int_{|x|\leq R_n}|u(t,x)|^{\frac{2N}{N-2}}dx=0.$$

\end{proposition}
\begin{proof}
  Note that by Cauchy-Schwartz, the uniform $H^1$ bounds for $u$, and the choice of weight function, we have
  $$\sup_{t\in\R}|M(t)|\lesssim_u R.$$

  We compute
  \begin{align}\label{Main-1}
    \frac{d}{dt}M(t)=&8\int_{|x|<R}|\nabla u|^2-\frac{B}{2p}(I_\alpha\ast|\cdot|^b|u|^p)|x|^b|u|^pdx\\\label{Eorr-1}
    &+\int_{|x|>2R}\frac{3R(N-1)(N-3)}{|x|^3}|u|^2+\frac{9R}{|x|}|\not\!\nabla u|^2dx\\\label{Eorr-2}
    &+\int_{R<|x|<2R}4{\rm Re}a_{jk}\bar{u}_ju_k+\mathcal{O}(\frac{R}{|x|^3}|u|^2)dx\\\label{Eorr-3}
    &+\mathcal{O}\left(\int_{|x|>R}(I_\alpha\ast|\cdot|^b|u|^p)|x|^b|u|^pdx\right),
  \end{align}
  where $\not\!\!\nabla=\nabla-\frac{x}{|x|^2}x\cdot\nabla$ denotes the angular part of the derivative. In \eqref{Eorr-1}, the angular derivation term is nonnegative, while the mass term is estimated by $R^{-2}$. Similarly, in \eqref{Eorr-2} the first term is nonnegative and the second term is estimated by $R^{-2}$.
  Using Hardy-Littlewood-Sobolev's inequality, we get
  $$\eqref{Eorr-3}\lesssim_E R^b\||x|^b\|_{L^{r_1}(A)}\|u\|_{L^r}^{2p},$$
the exponents satisfy the scaling relation,
$$1+\frac{\alpha}N=\frac{1}{r_1}+\frac{1}{r}.$$
Since $p\geq2$,  for every $A\in\{B_1(0),B_1^c(0)\}$, there exists $r\in[2,\frac{2N}{N-2}]$ such that $|x|^b\in L^{r_1}(A)$. Thus, \eqref{Eorr-3} can be estimated by $R^b$.

  We will make use the following identity, which can be checked by direct computation:
\begin{align}
  \int_{\R^N}\chi_R^2|\nabla u|^2=\int_{\R^N}\Big(|\nabla(\chi_Ru)|^2+\chi_R\Delta(\chi_R)|u|^2\Big)dx.
\end{align}
  In \eqref{Main-1} we may insert $\chi_R^2$, we can write
  \begin{align}
    \eqref{Main-1}=&8\left(\int_{\R^N}|\nabla(\chi_Ru)|^2dx-\frac{B}{2p}\int_{\R^N}(I_\alpha\ast|\cdot|^b|\chi_Ru|^p)
    |x|^b|\chi_Ru|^pdx\right)\\
    &+\int_{\R^N}\mathcal{O}\left(\frac{|u|^2}{R^2}\right)dx+\mathcal{O}\left(\int_{|x|>R}(I_\alpha\ast|\cdot|^b|u|^p)
    |x|^b|u|^pdx\right).
  \end{align}
 Continuing from above, we deduce
  $$\frac{1}{T}\int_{0}^T\int_{|x|<R}|u(t,x)|^\frac{2N}{N-2}dxdt\lesssim_{u,\delta} \frac{R}{T}+\frac{1}{R^{-b}}.$$
 Choosing $T$ sufficiently large and $T=R^{1-b}$ implies
 $$\frac{1}{T}\int_{\frac{T}{2}}^{T}\int_{|x|<T^{\frac1{1-b}}}|u(t,x)|^{\frac{2N}{N-2}}dxdt
 <T^{\frac{-b}{1-b}}.$$
 which suffices to give the desired result.
\end{proof}

\section*{Appendix: Morawetz type estimate}

 \noindent

  In this appendix, we consider the Morawetz estimate for NLS in $\R^N$. Supposing the function $u(t,x)$ solves $$i\pa_tu+\Delta u=F(t,x),\ \ (t,x)\in\R\times\R^N.$$
  Define Morawetz action
  $$M_a(t):=2{\rm Im}\int_{\R^N}\nabla a\cdot\nabla u\bar{u}dx.$$
  By a simple computation shows
  \begin{lemma}[Morawetz Identity \cite{Zh}]
    There holds
    $$\frac{d}{dt}M_a(t)=\int_{\R^N}(-\Delta\Delta a)|u|^2dx+4\int_{\R^N}a_{jk}{\rm Re}(\pa_j\bar{u}\pa_ku)dx+2\int_{\R^N}\nabla a(x)\cdot\{F,u\}_Pdx.$$
    where $\{f,g\}_P=Re(\bar{f}\nabla g-\bar{g}\nabla f)$. 
  \end{lemma}
  Next, we give a direction examples, i.e.,
  $F(t,x)=\lambda(I_\alpha\ast|\cdot|^b|u|^p)|x|^b|u|^{p-2}u,$ then
\begin{align}\nonumber
  2\int_{\R^N}\nabla a(x)\cdot\{F,u\}_Pdx=&\lambda\left(2-\frac{4}{p}\right)\int_{\R^N}\Delta a|u|^p(I_\alpha\ast|\cdot|^b|u|^p)|x|^bdx\\\nonumber
  &-\frac{4\lambda}p\int_{\R^N}\nabla a(x)\cdot\nabla[(I_\alpha\ast|\cdot|^b|u|^p)|x|^b]|u|^pdx\\\nonumber
  =&\lambda\left(2-\frac{4}{p}\right)\int_{\R^N}\Delta a|u|^p(I_\alpha\ast|\cdot|^b|u|^p)|x|^bdx\\\label{Main1}
  &-\frac{4\lambda}p\int_{\R^N}\nabla a\nabla(|x|^b)(I_\alpha\ast|\cdot|^b|u|^p)|u|^pdx\\\label{Main2}
  &-\frac{4\lambda}p\int_{\R^N}\nabla a((\nabla I_\alpha)\ast|\cdot|^b|u|^p)|x|^b|u|^pdx.
\end{align}
\textbf{The estimate of \eqref{Main1}:} By a simple computation, we have
$$\eqref{Main1}=\frac{-4b\lambda}p\int_{\R^N}\nabla a\cdot\frac{x}{|x|^2}(I_\alpha\ast|\cdot|^b|u|^p)|u|^pdx.$$
\textbf{The estimate of \eqref{Main2}:} Using the definition of $I_\alpha\ast|\cdot|^b|u|^p$, we obtain
\begin{align*}
\eqref{Main2}=&\frac{4\mathcal{K}(N-\alpha)\lambda}{p}\int_{\R^N}\int_{\R^N}\nabla a(x)\cdot\frac{x-y}{|x-y|^{N-\alpha+2}}|y|^b|u|^p|x|^b|u|^pdydx\\
=&-\frac{4\mathcal{K}(N-\alpha)\lambda}{p}\int_{\R^N}\int_{\R^N}\nabla a(y)\cdot\frac{x-y}{|x-y|^{N-\alpha+2}}|y|^b|u|^p|x|^b|u|^pdydx.
\end{align*}
Thus, we get
\begin{align*}
\eqref{Main2}=&\frac{2\mathcal{K}(N-\alpha)\lambda}{p}\int_{\R^N}\int_{\R^N}(\nabla a(x)-\nabla a(y))\cdot\frac{x-y}{|x-y|^{N-\alpha+2}}|y|^b|u|^p|x|^b|u|^pdydx.
\end{align*}
Hence, we get the Morawetz estimate
\begin{align*}
 \frac{d}{dt}M_a(t)=&\int_{\R^N}(-\Delta\Delta a)|u|^2dx+4\int_{\R^N}a_{jk}{\rm Re}(\pa_j\bar{u}\pa_ku)dx\\
 &+\lambda\left(2-\frac{4}{p}\right)\int_{\R^N}\Delta a|u|^p(I_\alpha\ast|\cdot|^b|u|^p)|x|^bdx\\
 &-\frac{4b\lambda}p\int_{\R^N}\nabla a\cdot\frac{x}{|x|^{2-b}}(I_\alpha\ast|\cdot|^b|u|^p)|u|^pdx\\
 &+\frac{2\mathcal{K}(N-\alpha)\lambda}{p}\int_{\R^N}\int_{\R^N}(\nabla a(x)-\nabla a(y))\cdot\frac{x-y}{|x-y|^{d-\alpha+2}}|y|^b|u|^p|x|^b|u|^pdydx.
\end{align*}

 Next, we choose $a(x)$ as in \eqref{weight-function}. Let $D:=\{x\in\R^N||x|\leq R\}\times\{y\in\R^N||y|\leq R\}\subset\R^{2N}$ and $D^{c}=\R^{2N}-D.$ Since $\nabla a(x)-\nabla a(y)=2(x-y)$ on $D$, then
\begin{align*}
  \frac{p}{2\mathcal{K}\lambda(N-\alpha)}\eqref{Main2}=&2\iint_{D}\frac{|y|^b|u(y)|^p|x|^b|u(x)|^p}{|x-y|^{N-\alpha}}dydx\\
  &+\iint_{D^c}(\nabla a(x)-\nabla a(y))\cdot\frac{x-y}{|x-y|^{N-\alpha+2}}|y|^b|u(y)|^p|x|^b|u(x)|^pdydx.
\end{align*}
\indent For the last integral, according to the definition of $a(x)$, we have $|\pa^{\beta}a(x)|\lesssim R^{2-|\beta|}(|\beta|\geq1)$ when $|x-y|<R$. At the same time, we can obtain
$$\left|(\nabla a(x)-\nabla a(y))\cdot\frac{x-y}{|x-y|^2}\right|\lesssim \|D^2a(x)\|_{L^\infty}.$$
\indent Else, when $|x-y|\geq R,$ then
$$\left|(\nabla a(x)-\nabla a(y))\cdot\frac{x-y}{|x-y|^2}\right|\lesssim \frac{\|D^2a(x)\|_{L^\infty}}{R}.$$
Moreover, by symmetry
$$\Big|\iint_{D^c}(\nabla a(x)-\nabla a(y))\cdot\frac{x-y}{|x-y|^{N-\alpha+2}}|y|^b|u(y)|^p|x|^b|u(x)|^pdydx\Big|\leq \frac2{\mathcal{K}}\int_{|x|>R}(I_\alpha\ast|\cdot|^b|u|^p)|x|^b|u(x)|^pdx.$$
Thus, We have the Morawetz estimate,
\begin{align*}\nonumber
 \frac{d}{dt}M_a(t)=&\int_{\R^N}(-\Delta\Delta a)|u|^2dx+4{\rm Re}\int_{\R^N}a_{jk}\pa_j\bar{u}\pa_kudx\\
 &+\lambda\bigg(2-\frac{4}{p}\bigg)\int_{\R^N}\Delta a|x|^b|u|^p(I_\alpha\ast|\cdot|^b|u|^p)dx\\\label{Main2-1add}
 &+\frac{4\mathcal{K}\lambda(N-\alpha)}{p}\iint_{D}\frac{|y|^b|u(y)|^p|x|^b|u(x)|^p}{|x-y|^{N-\alpha}}dydx\\
 &+\mathcal{O}\left(\int_{|x|>R}(I_\alpha\ast|\cdot|^b|u|^p)|x|^b|u|^pdx\right).
\end{align*}

\begin{center}

\end{center}
 \end{document}